\documentclass[11pt]{amsart}
\usepackage{amsfonts}
\usepackage{amsmath}
\usepackage{amsthm}
\usepackage{url}
\usepackage[all]{xy}
\usepackage{graphicx}
\usepackage{pdfsync}
\usepackage{latexsym}
\usepackage{amssymb}
\usepackage[cp850]{inputenc}
\usepackage{psfrag}
\usepackage{dsfont}
\usepackage{float}
\usepackage{stackrel}
\usepackage{hyperref}
\usepackage{tikz}

\usepackage{todonotes} 
\newcommand{\rdpnote}[1]{\todo[color=blue!10,linecolor=black,size= \tiny]{(rdp)-#1}}

\newtheorem{sat}{Theorem}[section]		
\newtheorem{lem}[sat]{Lemma}
			
\newtheorem{prop}[sat]{Proposition}

\newtheorem*{defi*}{Definition}			
\newtheorem*{bei*}{Example}
\newtheorem*{sat*}{Theorem}				
\newtheorem*{kor*}{Corollary}
\newtheorem*{rmk*}{Remark}				
	
\newtheorem*{quest*}{Question}

\let\ssection=\section
\renewcommand{\section}{\setcounter{equation}{0}\ssection}

\newtheorem*{namedtheorem}{\theoremname}
\newcommand{\theoremname}{testing}
\newenvironment{named}[1]{\renewcommand{\theoremname}{#1}\begin{namedtheorem}}{\end{namedtheorem}}

\theoremstyle{remark}
\newtheorem*{bem}{Remark}
\newtheorem{bei}{Example}

\newtheorem*{namedtheoremr}{\theoremnamer}
\newcommand{\theoremnamer}{testing}

\newcommand{\BC}{\mathbb C}			
			
			\newcommand{\BH}{\mathbb H}

			\newcommand{\BR}{\mathbb R}

			\newcommand{\BZ}{\mathbb Z}

			\newcommand{\CH}{\mathcal H}
\newcommand{\CI}{\mathcal I}			
			
\newcommand{\CM}{\mathcal M}		\newcommand{\CN}{\mathcal N}

\newcommand{\CS}{\mathcal S}			\newcommand{\CT}{\mathcal T}

\newcommand{\actson}{\curvearrowright}
\newcommand{\D}{\partial}
\newcommand{\DD}{\nabla}

\DeclareMathOperator{\Hom}{Hom}		
\DeclareMathOperator{\vol}{vol}		
\DeclareMathOperator{\tr}{Tr}

\DeclareMathOperator{\Map}{PMap}
\DeclareMathOperator{\MMap}{Map}

\DeclareMathOperator{\diam}{diam}

\newcommand{\comment}[1]{}

\DeclareMathOperator{\Stab}{Stab}

\DeclareMathOperator{\Homeo}{Homeo}
\DeclareMathOperator{\Aut}{Aut}

\DeclareMathOperator{\Bil}{Bil}

\DeclareMathOperator{\Tr}{Tr}

\DeclareMathOperator{\grad}{grad}
\DeclareMathOperator{\WP}{WP}
\DeclareMathOperator{\dist}{dist}

\newcommand{\fsubd}{\mathrel{{\scriptstyle\searrow}\kern-1ex^d\kern0.5ex}}
\newcommand{\bsubd}{\mathrel{{\scriptstyle\swarrow}\kern-1.6ex^d\kern0.8ex}}

\renewcommand{\epsilon}{\varepsilon}
\renewcommand{\le}{\leqslant}
\renewcommand{\ge}{\geqslant}
\renewcommand{\emptyset}{\varnothing}

\begin{document}

\title[]{Holomorphic maps between moduli spaces II}
\dedicatory{This paper is dedicated to Gabino Gonz\'alez-Diez.}
\author{Rodrigo De Pool}
\address{Instituto de Ciencias Matem\'aticas (ICMAT), Madrid, Spain}
\email{rodrigo.depool@icmat.es}
\author{Juan Souto}
\address{UNIV RENNES, CNRS, IRMAR - UMR 6625, F-35000 RENNES, FRANCE}
\email{jsoutoc@gmail.com}

\begin{abstract}
We prove that forgetful maps are the only non-constant holomorphic maps $\CM_{g,r}\to\CM_{g',r'}$ between moduli spaces, as long as $g\ge 4$ and $g'\le 3\cdot 2^{g-3}$.
\end{abstract}
\maketitle

\section{Introduction}

The goal of this note is to study holomorphic maps between moduli spaces. Denote by $\CT_{g,r}$ the Teichm\"uller space of Riemann surfaces of genus $g$ with $r$ marked points and by $\Map_{g,r}$ the pure mapping class group, that is, the group of isotopy classes of orientation preserving diffeomorphisms fixing each marked point. The action $\Map_{g,r}\actson\CT_{g,r}$ preserves the standard complex structure of Teichm\"uller space \cite{Nag}. We always think of moduli space as the complex orbifold
$$\CM_{g,r}=\CT_{g,r}/\Map_{g,r},$$
meaning that maps $F:M\to\CM_{g',r'}$ from a manifold, or more generally an orbifold, are induced by maps $\tilde F:\tilde M\to\CT_{g',r'}$ equivariant under a homomorphism $F_*:\pi_1(M)\to \Map_{g',r'}$. Here $\tilde M$ and $\pi_1(M)$ are the (orbifold) universal cover and fundamental group of $M$. We refer to \cite{Thurston} for more on orbifolds.

The basic examples of holomorphic maps between moduli spaces are the {\em forgetful maps}, that is, the maps $\CM_{g,r}\to\CM_{g,r-k}$ obtained by unmarking $k\le r$ marked points. In fact, it was proved in \cite{griego} that forgetful maps are, as long as $g\ge 6$ and $g'\le 2g-2$, the only non-constant holomorphic maps $F:\CM_{g,r}\to\CM_{g',r'}$ between moduli spaces. The goal of this paper is to show that the same holds true when we replace the linear bound on $g'$ by an exponential bound:

\begin{sat}\label{main}
Suppose that $g\ge 4$ and that $g'\le 3\cdot 2^{g-3}$, and let $r,r'\ge 0$. Then every non-constant holomorphic map $F:\CM_{g,r}\to\CM_{g',r'}$ is a forgetful map. In particular, if such a map exists, then $g'=g$ and $r'\le r$.
\end{sat}

Note that some condition on $g$ is needed in Theorem \ref{main}: the holomorphic map $\CM_{2,0}\to\CM_{0,6}$ obtained by quotienting by the hyperelliptic involution is not a forgetful map. Similarly, some exponential bound for $g'$ is also needed: in Example \ref{example covers} in the last section of this paper we present a non-constant holomorphic map $\CM_{g,1}\to\CM_{(g-1)\cdot 2^{2g}-1,1}$ constructed via covers which is not a forgetful map.
\medskip

Let us briefly recall the strategy followed in \cite{griego}. First, one derives from \cite{Javi-Juan} that if $g\ge 6$ and $g'\le 2g-2$ then every non-trivial homomorphism $\Map_{g,r}\to \Map_{g',r'}$ is induced by forgetting marked points. It follows that, still in the same range, every non-constant holomorphic map $\CM_{g,r}\to\CM_{g',r'}$ is homotopic to a forgetful map. The claim then follows from the fact that moduli space is, as an analytic space, a quasi-projective variety \cite{Deligne-Mumford}, and that any two non-constant homotopic holomorphic maps from a quasi-projective variety to moduli space agree \cite[Prop 3.2]{Javi-Juan}.

The strategy of the proof of Theorem \ref{main} basically follows the same lines. The first and main ingredient is a theorem by the first author \cite{Rodrigo} classifying all non-trivial homomorphisms $\rho:\Map_{g,r}\to \Map_{g',r'}$ for $g,g'$ as in Theorem \ref{main}. The new difficulty we encounter is that it is no longer true that all such homomorphisms are induced by forgetful maps---see Example \ref{example multi-embedding}. We get however that all {\em irreducible} homomorphisms are induced by forgetful maps, where irreducible means that there is no  simple (non-trivial) multicurve $\gamma\subset\Sigma_{g',r'}$ which is fixed by the image $\rho(\Map_{g,r})$. What we need to prove is that homomorphisms $F_*$ induced by holomorphic maps $F:\CM_{g,r}\to\CM_{g',r'}$ are indeed irreducible. 

\begin{sat}\label{thm irreducible}
If $M$ is an irreducible quasi-projective variety and $F:M\to\CM_{g,r}$ is a non-constant holomorphic map, then the homomorphism $F_*:\pi_1(M)\to \Map_{g,r}$ is irreducible.
\end{sat}

Let us be clear about the terminology in Theorem \ref{thm irreducible}. Per se, a quasi-projective variety is nothing other than a Zariski open subset of a projective variety over some field, which we will always assume to be $\BC$. We can thus consider quasi-projective varieties as analytic spaces, referring to analytic maps between such spaces as holomorphic maps. For example, it is due to Deligne and Mumford \cite{Deligne-Mumford} that moduli space $\CM_{g,r}$ is a quasi-projective variety. Moreover, since we are considering the moduli space as an orbifold (that is, we are considering it as a fine moduli space, not as a coarse one), we also want to allow the case that our variety $M$ is an orbifold and that the map $F:M\to\CM_{g,r}$ is a map between orbifolds. Let us stress this point: the quasi-projective variety $M$ in Theorem \ref{thm irreducible} is allowed to be orbifold. If $M$ is an orbifold, then we assume that the holomorphic map $F:M\to\CM_{g,r}$ is induced by a holomorphic map $\tilde F:\tilde M\to \CT_{g,r}$ which is equivariant under a homomorphism $F_*:\pi_1(M)\to\Map_{g,r}$, where $\tilde M$ is the orbifold universal cover $\pi_1(M)$ the orbifold fundamental group.
\medskip

Once the statement of Theorem \ref{thm irreducible} is clarified, let us add a few comments on its proof. First note that since every quasi-projective variety of complex dimension at least 2 contains a wealth of algebraic curves, it suffices to prove Theorem \ref{thm irreducible} in the case that the domain $M$ is 1-dimensional, and more specifically a Riemann surface of finite analytic type. If the domain of $F:M\to\CM_{g,r}$ is closed, something we cannot assume here because moduli space is not closed, then the irreducibility of $F_*$ is due to McMullen \cite{McMullen}. His basic observation is that if $F_*(\pi_1(M))$ were to fix a multicurve $\gamma$ then the map $F:M\to\CM_{g,r}$ would lift to a map $F':M\to\CT_{g,r}/\Stab(\gamma)$.
The function $\ell_\gamma$ associating to a point in Teichm\"uller space $X$ the length of $\gamma$ with respect to the hyperbolic structure of $X$ descends to a well-defined function $\hat\ell_\gamma:\CT_{g,r}/\Stab(\gamma)\to\BR_{>0}$. Wolpert \cite{Wolpert geodesically convex} has proved that $\ell_\gamma$ is plurisubharmonic, and from this fact McMullen obtains that $\hat\ell_\gamma\circ F'$ would be constant. He then gets a contradiction from what he describes, without overwhelming the reader with unnecessary details, as `an argument with quasifuchsian groups' \cite[p. 136]{McMullen}.

In the setting we care mostly about, that is the case that the domain $M$ is not compact, we cannot argue like that: plurisubharmonicity is not enough to show that $\hat\ell_\gamma\circ F'$ would be constant. Indeed, neither plurisubharmonicity nor quasifuchsian groups will play any role here. For our argument we rely directly on  what one can think are the essential qualities of the Weil-Petersson geometry: $(\CT_{g,r},g_{\WP})$ is K\"ahler, incomplete, negatively curved, geodesically convex, and the metric is dominated by the Kobayashi metric, that is the Teichm\"uller metric. Consider namely the function $h_\gamma:\CT_{g,r}\to\BR_{>0}$ sending $X\in\CT_{g,r}$ to the square of its distance to the stratum $\CS_\gamma\subset\overline\CT_{g,r}$ consisting of those points with nodal locus $\gamma$. Negative curvature implies that $h_\gamma$ is strictly convex. Moreover, when it exists, the norm of its differential $dh_\gamma\vert_X$ is at most lineal in the distance $d_{\WP}(X_0,X)$ to some arbitrary base point $X_0$. This implies the function $\Vert dh_\gamma\vert_X\Vert$  is subexponential and  so we say that $h_\gamma$ has {\em subexponential growth}. We get thus a contradiction from the following result:

\begin{sat}\label{thm non-existence}
Suppose that $N$ is a K\"ahler manifold whose metric is dominated by a multiple of the Kobayashi metric and which admits a strictly convex function with subexponential growth. Let also $M$ be either a closed connected K\"ahler manifold or an irreducible quasi-projective variety. Then there are no non-constant holomorphic maps $F:M\to N$.
\end{sat}

We stress that we are neither assuming that $N$ is complete, nor that the convex function $f$ is smooth. Let us however suppose for a moment that $f$ is smooth. The basic idea of the proof of Theorem \ref{thm non-existence} is that convexity implies that flowing any non-constant map $F:M\to N$ in the direction of the negative gradient of $f$ reduces the energy of $F$. On the other hand, the standard Wirtinger inequality implies that holomorphic maps are absolute energy minimizers in their homotopy class. A contradiction.
\medskip

This paper is organized as follows. After recalling some basic properties of Teichm\"uller space and the Weil-Petersson metric in Section \ref{sec preli teich}, we prove Theorem \ref{thm non-existence} in section \ref{sec irreducible}. Theorem \ref{thm irreducible} and Theorem \ref{main} are dealt with in Section \ref{sec main}.

\subsection*{Acknowledgements}
The first author acknowledges financial support from the grant CEX2019-000904-S funded by MCIN/AEI/10.13039/501100011033 and from the grant PGC2018-101179-B-I00.

\section{Teichm\"uller space and the Weil-Petersson metric}\label{sec preli teich}
Throughout this section, let us fix $g,r\ge 0$ satisfying $2g+r\ge 3$. Let us also denote by $S_{g,r}$ the smooth oriented surface of genus $g$ with $r$ ends and empty boundary.

\subsection{Preliminaries}
Recall that a Riemann surface $\Sigma$ is of finite analytic type if it is biholomorphic to the complement of finitely many points in a closed Riemann surface. As long as it has negative Euler-characteristic, $\Sigma$ admits a unique conformal complete hyperbolic metric (that is, of constant curvature $-1$) which automatically has finite volume. Here, conformal means that, for all $x\in\Sigma$, the scalar product induced on $T_x\Sigma$ is invariant under multiplication by the imaginary unit $i\in\BC$.

Under a {\em marked Riemann surface} we understand a pair $(\Sigma,\phi:S_{g,r}\to\Sigma)$ where $\phi$ is a orientation preserving homeomorphism. The {\em Teichm\"uller space} $\CT_{g,r}$ is the space of all equivalence classes of marked Riemann surfaces, where $(\Sigma,\phi:S_{g,r}\to\Sigma)$ and $(\Sigma',\phi':S_{g,r}\to\Sigma')$ are equivalent if and only if $\phi'\circ\phi^{-1}:\Sigma\to\Sigma'$ is isotopic to a biholomorphism.

The mapping class group $\MMap_{g,r}=\Homeo_+(S_{g,r})/\Homeo_0(S_{g,r})$ is the group of isotopy classes of orientation preserving self-homeomorphisms of $S_{g,r}$. It acts on the ends of $S_{g,r}$. The {\em pure mapping class group} $\Map_{g,r}$ is the finite index subgroup fixing each end. 

The mapping class group also acts on Teichm\"uller space, and the quotient
$$\CM_{g,r}=\CT_{g,r}/\Map_{g,r}$$
under the action of the pure mapping class group is the {\em moduli space}. Dividing by the whole mapping class group amounts to forgetting the marking. Since we are dividing by the pure mapping class group, the quotient $\CM_{g,r}$ is the space of biholomorphism classes of Riemann surfaces of genus $g$ with $r$ labelled points. Teichm\"uller space admits a structure as simply connected complex manifold, and the mapping class group acts on it by biholomorphisms. Hence, $\CM_{g,r}$ can be seen as a complex orbifold with universal cover $\CT_{g,r}$. It is well known that $\CM_{g,r}$ admits an algebraic structure. Indeed, $\CM_{g,r}$ is an irreducible quasi-projective variety, that is a Zariski open subset of a projective variety \cite{Deligne-Mumford}.

Recall now that the {\em Kobayashi metric} of a complex manifold $M$ is the largest pseudo-metric with the property that every holomorphic map $F:\BH^2\to M$ is 1-Lipschitz, where $\BH^2$ is the hyperbolic plane endowed with the hyperbolic metric. In the case of Teichm\"uller space, the Kobayashi metric is an actual metric, which moreover agrees with the well-known Teichm\"uller metric \cite{Royden}. We refer to \cite{Farb-Margalit} and \cite{Imayoshi-Taniguchi} for more on mapping class groups and Teichm\"uller spaces, and to \cite{Kobayashi} for more on the Kobayashi metric on general complex manifolds.

\subsection{The Weil-Petersson metric}
Teichm\"uller space admits many natural mapping class group invariant metrics---we already mentioned the Teichm\"uller metric. Here we will be working with the so-called Weil-Petersson metric $g_{\WP}$. This Riemannian metric was introduced by Weil in \cite{Weil}, and Ahlfors \cite{Ahlfors} proved that it is K\"ahler. Since this property will be key in this paper, let us recall that a K\"ahler metric on a complex manifold $N$ is nothing other than a Riemannian metric $\langle\cdot,\cdot\rangle$ with $\langle iX_p,iY_p\rangle_p=\langle X_p,Y_p\rangle_p$ for all $p\in N$ and $X_p,Y_p\in T_pN$, and so that the K\"ahler form 
$$\omega_p(X_p,Y_p)=\langle X_p,iY_p\rangle_p$$
is closed: $d\omega=0$. Note in particular that every conformal metric on a Riemann surface is automatically K\"ahler. We refer to for example \cite{Ballman} for facts about K\"ahler manifolds.

A first key fact about the Weil-Petersson metric is that it is dominated, up to a factor $\vert 2\pi\chi(S_{g,r})\vert^{\frac 12}$, by the Teichm\"uller metric (see for example \cite[Prop. 2.4]{McMullen Kahler}). Since the Teichm\"uller metric agrees with the Kobayashi metric, we have:

\begin{lem}\label{lem 1}
The Weil-Petersson metric on $\CT_{g,r}$ is dominated by a multiple of the Kobayashi metric.\qed
\end{lem}

From a metric point of view, what was maybe first noted about the Weil-Petersson metric is that it is incomplete \cite{Wolpert incomplete}. To see that this is the case, Wolpert proved that muticurves in $S_{g,r}$ can be pinched in finite time. Indeed, Yamada \cite{Yamada} identified the completion $\overline\CT_{g,r}$ of $(\CT_{g,r},g_{\WP})$ with what is known as {\em augmented Teichm\"uller space}. In other words, the completion of $(\CT_{g,r},g_{\WP})$ consists of {\em nodal curves}, that is Riemann surfaces where some multicurve has been pinched to have length $0$. Given a simple (always non-trivial) multicurve $\gamma$ in $S_{g,r}$, we denote by
\begin{equation}\label{eq set S}
\CS_\gamma=\{X\in\overline\CT_{g,r}\text{ with } \ell_X(\gamma)=0\}
\end{equation}
the (closed) subset of the completion of $(\CT_{g,r},g_{\WP})$ where the components of $\gamma$ have been pinched. We stress the fact that $\CS_\gamma\neq\emptyset$.

From a Riemannian point of view, what is maybe best known about the Weil-Petersson metric is that its sectional curvature is negative, a fact due to Royden, Wolpert \cite{Wolpert negative} and Tromba \cite{Tromba WP negative}. However, incomplete Riemannian manifolds, even negatively curved ones, can in general be pretty nasty to work with. In this case one is saved by a landmark result of Wolpert \cite{Wolpert geodesically convex} asserting that $(\CT_{g,r},g_{WP})$ is a geodesic metric space, or in Riemannian terms that $g_{\WP}$ is geodesically convex in the sense that any two points are joined by a Weil-Petersson geodesic. In the same paper, he proved that the length function $\ell_\gamma:\CT_{g,r}\to(0,\infty)$ is convex, meaning that $\ell_\gamma\circ\alpha:(a,b)\to\BR$ is convex for any Weil-Petersson geodesic $\alpha:(a,b)\to\CT_{g,r}$.

As a consequence of the convexity of $\ell_\gamma$ and of the geodesic convexity of $\CT_{g,r}$ one gets that the latter is a nested union of compact convex sets: pick two filling curves $\gamma_1$ and $\gamma_2$, that is two curves which have positive geometric intersection number with every curve, and consider compact sets of the form 
\begin{equation}\label{eq good convex}
K(L)=\{X\in\CT_{g,r}\text{ with }\ell_{\gamma_1}(X)+\ell_{\gamma_2}(X)\le L\}.
\end{equation}
Since the sectional curvature is overall negative, one has that on $K(L)$ it is bounded from above by some constant $\kappa_L<0$. This implies that $K(L)$ is CAT($\kappa_L$), and hence that $\CT_{g,r}=\cup_LK(L)$ is CAT(0). Since being CAT(0) is a property inherited by metric completions of geodesic metric spaces, we get that $\overline\CT_{g,r}$ is also CAT(0). As a consequence, any two points in the completion $\overline\CT_{g,r}$ of $(\CT_{g,r},g_{\WP})$ are joined by a unique geodesic segment. We refer to \cite{Bridson-Haffliger} for facts on CAT(0)-spaces.

\subsection{A strictly convex function}
Convex functions, in fact strictly convex functions, play a key role in this note.

\begin{defi*}
A continuous function $f\in C^0(M)$ on a Riemannian manifold is {\em strictly convex}, if for every compact set $K\subset M$ there are $\delta,C>0$ such that for any $\epsilon<\delta$ and geodesic segment $\alpha:[-\epsilon,\epsilon]\to M$ parameterized by arc length and with $\alpha(0)\in K$ we have $f(\alpha(\epsilon))+f(\alpha(-\epsilon))-2f(\alpha(0))\ge C\cdot\epsilon^2$.
\end{defi*}

\begin{bem}
Note that a smooth function $f\in C^\infty(M)$ is strictly convex if and only if we have $\frac{d^2}{dt^2}(f\circ\alpha)\vert_{t=0}>0$ for every geodesic $\alpha:(\epsilon,\epsilon)\to M$ parameterized by arc length.
\end{bem}

Fixing a simple multicurve $\gamma\subset S_{g,r}$ let $\CS_\gamma\subset\overline\CT_{g,r}$ be as in \eqref{eq set S} the set of points in the completion where $\gamma$ has been pinched. Convexity of $\ell_\gamma$ on $\CT_{g,r}$ implies that the set $\CS_\gamma$ is convex in $\overline\CT_{g,r}$. Now, since $\overline\CT_{g,r}$ is CAT(0), and since in such a space the distance to a convex set is a convex function, we get that
$$\dist_\gamma:\CT_{g,r}\to\BR_{>0},\ \dist_\gamma(X)=\min_{Y\in\CS_\gamma}d_{\overline\CT_{g,r}}(X,Y)$$
is convex. We claim that its square is strictly convex:

\begin{lem}\label{lem 3}
Let $\gamma\subset S_{g,r}$ be a simple multicurve. The function 
\begin{equation}\label{eq height function}
h_\gamma:\CT_{g,r}\to\BR_{>0},\ h_\gamma(X)=\min_{Y\in\CS_\gamma}d_{\overline\CT_{g,r}}(X,Y)^2
\end{equation}
is strictly convex.
\end{lem}

It is well-known that in a CAT(0)-space, the square of the distance function to a given point is strictly convex. This is however not true for the square of the distance function to a general convex set---think of the square of the distance function to a line in $\BR^2$. This is why we have to work a bit to prove Lemma \ref{lem 3}. The proof will rely on a standard argument in comparison geometry together with the elementary fact that if $a,b,c,d\in\BH^2$ are 4 points such that the geodesic segments $[a,b]$ and $[c,d]$ are non-degenerate and disjoint, then the function $d_{\BH^2}(\cdot,[c,d])^2$ is strictly convex along $[a,b]$. Although this seems pretty much evident and is surely well-known, we didn't find a handy reference. Thus, we give an argument. Well, note that every point in $[a,b]$ is contained in a closed interval $[a',b']\subset[a,b]$ on which $d_{\BH^2}(\cdot,[c,d])^2$ agrees with one of the three functions
\begin{equation}\label{eq estoy hasta los huevos}
p\mapsto d_{\BH^2}(p,c)^2,\ p\mapsto d_{\BH^2}(p,d)^2,\text{ or }p\mapsto d_{\BH^2}(p,\gamma)^2
\end{equation}
where $\gamma$ is the unique bi-infinite geodesic containing $[c,d]$---to get a clean picture note that in the last case we have that $[a',b']\cap\gamma=\emptyset$. Anyways, since the first two functions in \eqref{eq estoy hasta los huevos} are strict convex and since $d_{\BH^2}(\cdot,[c,d])^2$ is $C^1$, it suffices to argue that the function $t\mapsto d_{\BH^2}(\alpha(t),\gamma)^2$ is strictly convex whenever $t\to\alpha(t)$ parameterizes by arc length an (open) geodesic segment disjoint of $\gamma$. Since this function is smooth, it suffices to prove that its second derivative
$$2\left(\frac d{dt}d_{\BH^2}(\alpha(t),\gamma)\right)^2+2\cdot d_{\BH^2}(\alpha(t),\gamma) \cdot\frac{d^2}{dt^2}d_{\BH^2}(\alpha(t),\gamma)$$
is positive. Well, the second summand is non-negative because the distance function is convex and the first is positive unless $d_{\BH^2}(\alpha(t),\gamma)$ attains its minimum at $t$. Now, when we are in the minimum, it follows from a calculation (using for example \cite[p.454, 2.3.1 (v)]{Buser}) that the second factor is positive at that $t$. Having proved the strict convexity of $h$, we get that also the third function in \eqref{eq estoy hasta los huevos} is strict convex, and hence that the function $d_{\BH^2}(\cdot,[c,d])^2$ itself is strict convex on $[a,b]$. In fact, since everything, including lower bounds for the second derivatives, depend smoothly on $a,b,c,d$ we get that there is a constant $C$ which depends only on upper and lower bounds\footnote{Actually, a lower bound suffices.} for the distance to $[c,d]$ from points in $[a,b]$ so that we have 
\begin{equation}\label{eq no te lo puedes ni imaginar}
d_{\BH^2}(p^-,[c,d])^2+d_{\BH^2}(p^+,[c,d])^2-2\cdot d_{\BH^2}(p^0,[c,d])^2\ge C\cdot d_{\BH^2}(p^-,p^0)
\end{equation}
for all tuples $(p^-,p^0,p^+)$ of distinct points in $[a,b]$ with $d_{\BH^2}(p^-,p^0)=d_{\BH^2}(p^0,p^+)$. 

Well, after all this preparation, we are ready to prove Lemma \ref{lem 3}.

\begin{proof}[Proof of Lemma \ref{lem 3}]
Suppose that $K\subset\CT_{g,r}$ is compact, let $\delta_0>0$ be so that the distance between $K$ and $\overline\CT_{g,r}\setminus\CT_{g,r}$ is at least $10\delta_0$, and let $\CN_{\delta_0}(K)$ be the set of points at distance at most $\delta_0$ of $K$. Finally, fix $L>0$ so that the compact convex set $K(L)$ from \eqref{eq good convex} contains both $\CN_{\delta_0}(K)$ and every point $y\in\CT_{g,r}$ with $d_{\CT}(y,\CS_\gamma)=\delta_0$ and which lies in a minimizing geodesic segment joining $\CS_\gamma$ to some $x\in\CN_{\delta_0}(K)$. All points considered below will belong to $K(L)$.

Anyways, suppose now that we have a geodesic
$$\alpha:[-\delta_0,\delta_0]\to\CT_{g,r}$$ 
parameterized by arc length and with $\alpha(0)\in K$ and take $\epsilon\le\delta_0$.

\begin{figure}[h]
\begin{center}
\begin{tikzpicture}

\node[anchor=south west,inner sep=0] at (0,0) {\includegraphics[width=0.5\linewidth]{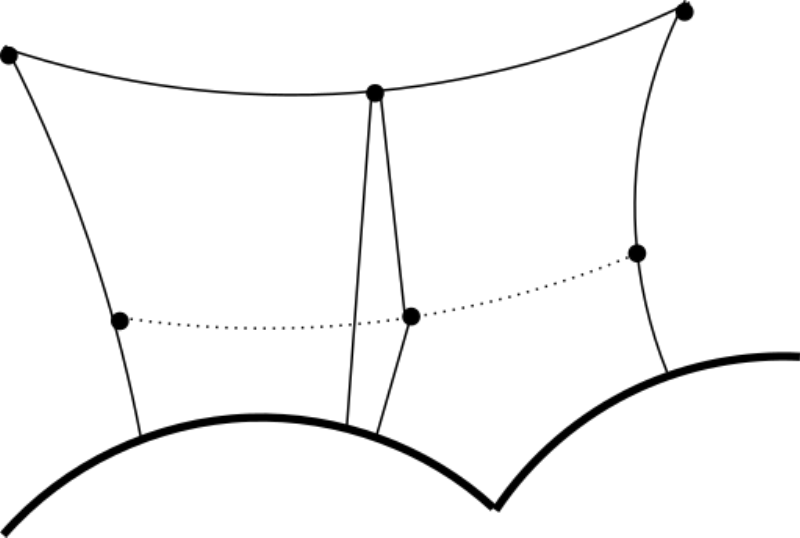}};

\node[label=above right:{$x^-=\alpha(-\epsilon)$}] at (-1.9,3.7){};
\node[label=above right:{$\alpha(0)$}] at (2.1,3.4){};
\node[label=above right:{$x^+=\alpha(\epsilon)$}] at (5.1,4){};
\node[label=above right:{$y^-$}] at (-0.1,1.2){};
\node[label=above right:{$m$}] at (3,1.1){};
\node[label=above right:{$y^+$}] at (4.9,1.5){};
\end{tikzpicture}

\end{center}
\caption{The bold printed lines represent $\CS_\gamma$. The lines joining $x^{\pm}$ to $\CS_\gamma$ represent the segments $I^\pm$. The points $y^{\pm}$ are chosen so that $d_{\CT}(y^{\pm},\CS_\gamma)=\delta_0$ and the point $m$ is the projection of $\alpha(0)$ to $[y^{-1},y^+]$. Since $\overline\CT_{g,r}$ is CAT(0) we have that $d_{\CT}(m,\CS_\gamma)\le \delta_0$.}
\label{fig 1}
\end{figure}
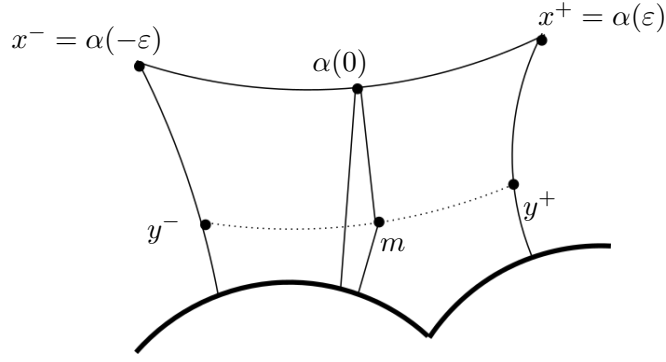

As in Figure \ref{fig 1} consider the shortest geodesic segments $I^-,I^+$ joining $x^-=\alpha(-\epsilon)$ and $x^+=\alpha(\epsilon)$ to $\CS_\gamma$. Let also $y^{\pm}\in I^\pm$ be the point at distance $\delta_0$ from $\CS_\gamma$, and finally let $m\in[y^-,y^+]$ be the point in there closest to $\alpha(0)$ (compare with Figure \ref{fig 1}). Convexity in the CAT(0)-space $\overline\CT_{g,r}$ of the distance to the convex set $\CS_\gamma$ implies that 
$$d_{\CT}(m,\CS_\gamma)\le\frac 12\left(d_{\CT}(y^-,\CS_\gamma)+d_{\CT}(y^\pm,\CS_\gamma)\right)=\delta_0$$
and hence that
$$d_{\CT}(\alpha(0),\CS_\gamma)\le d_{\CT}(\alpha(0),m)+d_{\CT}(m,\CS_\gamma)\le d_{\CT}(\alpha(0),m)+\delta_0.$$
Convexity of the distance function on $\CT_{g,r}$ implies also that
$$d_{\CT}(x^-,y^-)+d_{\CT}(x^+,y^+)-2\cdot d_\CT(\alpha(0),m)\ge 0.$$
Multiplying out and invoking the last two displayed equations we get that 
\begin{align*}
h(\alpha&(-\epsilon)+h(\alpha(\epsilon))-2h(\alpha(0))\ge\\
&= d_\CT(x^-,\CS_\gamma)^2+d_\CT(x^+,\CS_\gamma)^2-2\cdot d_\CT(\alpha(0),\CS_\gamma)^2\\
&\ge (d_\CT(x^-,y^-)+\delta_0)^2+(d_\CT(x^+,y^+)+\delta_0)^2-2\cdot(d_\CT(\alpha(0),m)+\delta_0)^2\\
&\ge d_\CT(x^-,y^-)^2+d_\CT(x^+,y^+)^2-2\cdot d_\CT(\alpha(0),m)^2
\end{align*}

Recall at some point that all of this is happening is a fixed compact convex set $K(L)$ as in \eqref{eq good convex}---more specifically that $x^\pm,y^\pm\in K(L)$. Let then $\kappa<0$ be an upper bound for the sectional curvature on $K(L)$. As we mentioned earlier, convexity of $K(L)$ and the upper bound on the curvature imply that $K(L)$ is a CAT($\kappa$)-space. It follows in particular that distances between points in the boundary of the triangles $(x^-,y^-,y^+)$ and $(x^-,y^+,x^+)$ are smaller than in the comparison triangle $(\bar x^-,\bar y^-,\bar y^+)$ and $(\bar x^-,\bar y^+,\bar x^+)$ in the space $\BH^2_{\kappa}$ of constant curvature $\kappa$. When we glue, as suggested by the notation, those two comparison triangles along the side $[\bar x^-,\bar y^+]$ we get a quadrilateral satisfying that 
\begin{align*}
d_{\CT}&(x^\pm,y^\pm)=d_{\BH^2_\kappa}(\bar x^\pm,\bar y^\pm)\text{ and}\\
d_{\CT}&(\alpha(0),m)\le d_{\BH^2_{\kappa}}(\text{mid}(\bar x^-,\bar x^+)),[\bar y^-,\bar y^+]).
\end{align*}
We thus have that
\begin{multline*}
h(\alpha(-\epsilon)+h(\alpha(\epsilon))-2h(\alpha(0))\ge\\
\ge d_{\BH^2_{\kappa}}(\bar x^-,\bar y^-)^2+d_{\BH^2_{\kappa}}(\bar x^+,\bar y^+)^2-2d_{\BH^2_{\kappa}}(\text{mid}(\bar x^-,\bar x^+),[\bar y^-,\bar y^+])^2
\end{multline*}
Noting now that, by construction and the CAT($\kappa$) property, the segments $[\bar x^-,\bar x^+]$ and $[\bar y^-,\bar y^+]$ are at least at distance $8\delta_0$ of each other, and that their distance is bounded from above by $\diam_\CT(K(L))$, we get that the right side of the last inequality is bounded from below by $C_K\cdot\epsilon^2$ for some $C_K>0$ which depends only on $K(L)$, and hence indirectly only on $K$. Putting things together, we get that 
$$h(\alpha(-\epsilon)+h(\alpha(\epsilon))-2h(\alpha(0))\ge C_K\cdot\epsilon^2\text{ for all }\epsilon\in(0,\delta_0).$$
Since $\alpha:(-\delta_0,\delta_0)\to\CT_{g,r}$ was an arbitrary geodesic parameterized by arc length with $\alpha(0)\in K$, we have proved what we had to prove. 
\end{proof}

Continuing with the properties of the function $h_\gamma$ given by \eqref{eq height function} note that that distance function $\dist_\gamma(\cdot)$ is Lipschitz with Lipschitz constant 1. It follows from Rademacher's theorem that $\dist_\gamma(\cdot)$ is almost everywhere differentiable with gradient of norm $1$. From here we get that the function $h_\gamma$ is almost everywhere differentiable with gradient of norm $2\cdot\dist_\gamma(\cdot)$, meaning that it has what we call {\em subexponential growth}:

\begin{defi*}
We will say that an almost everywhere differentiable function $f:M\to\BR$ on a connected Riemannian manifold $M$ {\em has subexponential growth} if for some, and hence any point $p_0\in M$ we have
$$\lim_{\tiny{\begin{array}{l}M\ni p\to \infty\\ df_p\text{ exists}\end{array}}}\Vert df_p\Vert\cdot e^{-\epsilon\cdot d_M(p,p_0)}=0$$
for all $\epsilon>0$.
\end{defi*}

Finally, note that $\CS_\gamma$ is $\Stab_{\Map_{g,r}}(\gamma)$-invariant, and that this implies that our function $h_\gamma$ is also $\Stab_{\Map_{g,r}}(\gamma)$-invariant. We collect all the properties we have established for $h_\gamma$ in the following statement:

\begin{prop}\label{prop 3}
Let $\gamma\subset S_{g,r}$ be a simple multicurve. The function $h_\gamma:\CT_{g,r}\to\BR_{>0}$ given by \eqref{eq height function} is $\Stab_{\Map_{g,r}}(\gamma)$-invariant, strictly convex and has subexponential growth.\qed
\end{prop}

Lemma \ref{lem 1} and Proposition \ref{prop 3}, together with the fact that $g_{\WP}$ is K\"ahler, encapsulate everything we will need about the Weil-Petersson metric in this paper. We refer however to \cite{Wolpert1,Wolpert2,Wolpert3} for more about the Weil-Petersson geometry of Teichm\"uller space.

\section{Proof of Theorem \ref{thm non-existence}}\label{sec irreducible}
We devote the first 3 parts of this section to recall a few facts on Riemannian geometry, more specifically on gradient flows of smooth convex functions and on the fact that holomorphic maps are critical points of the energy functional. These results are well-known, and readers familiar with  tools from complex geometry may choose to skim over the statements before seeing how these ingredients are combined to prove Theorem \ref{thm non-existence}. 
\medskip

Throughout this section, $N$ will be a Riemannian manifold, which we will eventually assume to be K\"ahler, and a fortiori complex. We denote the Riemannian metric by $\langle\cdot,\cdot\rangle=\langle\cdot,\cdot\rangle_M$ and the associated Levi-Civita connection by $\DD$. We stress that we are not assuming that our manifold is complete.

A comment on notation. By $X_p,\langle\cdot,\cdot\rangle_p,\omega_p$ and such we denote a tangent vector at $p\in M$, the scalar product or an alternating form on $T_pM$. When these pointwise objects are the restriction of a global object such as a vector field $X\in\Gamma(TM)$, and when we want to stress its evaluation at $p$, then we might write $X\vert_p$ instead of simply $X_p$.

\subsection{Convexity}
Recall that the {\em Hessian} at a point $p\in N$ of a smooth function $f\in C^\infty(N)$ is the symmetric bilinear form given by
\begin{align*}
Hf\vert_p&:T_pM\times T_pM\to\BR\\
Hf\vert_p&(X_p,Y_p)=\langle\DD_X\grad(f),Y\rangle\vert_p,
\end{align*}
where $\grad(f)$ is the gradient of $f$ and where $X,Y\in\Gamma(TM)$ are arbitrary smooth vector fields on $M$ with $X\vert_p=X_p$ and $Y\vert_p=Y_p$. The reason we are interested in the Hessian is that it detects convexity of smooth  functions. Indeed, noting that for every geodesic $\alpha:(-\epsilon,\epsilon)\to M$ we have
\begin{equation}\label{eq hessian}
\begin{split}
\frac{d^2}{dt^2}(f\circ\alpha)
&=\frac d{dt}\langle\grad(f),\alpha'(t)\rangle=\langle\DD_{\alpha'(t)}\grad(f),\alpha'\rangle+\langle\DD f,\DD_{\alpha'(t)}\alpha'(t)\rangle\\
&=\langle\DD_{\alpha'(t)}\grad(f),\alpha'\rangle=Hf\vert_{\gamma(t)}(\alpha'(t),\alpha'(t)),
\end{split}
\end{equation}
we get:

\begin{lem}\label{lem 4}
Let $M$ a Riemannian manifold. A smooth function $C^\infty(M)$ is strictly convex if and only if its Hessian $Hf\vert_p$ is positive definite at every $p\in M$. \qed
\end{lem}

Another well-known fact is that the Hessian of a function controls the metric behavior of the flow associated to its negative gradient. We make this formal in the setting we will need it:

\begin{lem}\label{lem distortion flow}
Let $F:M\to N$ be a smooth map between Riemannian manifolds, let $f\in C^\infty(N)$ and $\rho\in C^\infty_c(M)$ be two smooth functions, the latter non-negative and with compact support, and for $t\in\BR$ set 
$$F_t:M\to N,\ F_t(x)=\phi_{\rho(x)\cdot t}(F(x))$$
where $(\phi_t)$ is the local flow of $-\grad f$. We have
$$\frac d{dt}\Vert dF_t\vert_xX_x\Vert^2\vert_{t=0}=-2\rho (Hf\vert_{F(x)})(dF_xX_x,dF_xX_x)-2d\rho(X_x)\cdot df_{F(x)}(dF_xX_x)$$
for every $x\in M$ and $X_x\in T_xM$.
\end{lem}

\begin{proof}
Before launching the proof, note that, while $\phi_t(p)$ might be only defined for $t\in(\epsilon_p,\epsilon_p)$ for some $\epsilon_p>0$ depending on $p\in N$, the fact that $\rho$ has compact support implies that there is some $\epsilon>0$ such that $F_t:M\to N$ is defined for all $t\in(-\epsilon,\epsilon)$. This means in particular that the claim in the lemma does actually make sense as stated. Now, once we have said this we can forget about the condition that $\rho$ is compactly supported and since the statement is purely local we can replace $M$ by a small neighborhood of the point $x$ we are interested in. In particular, we may assume $\rho>0$. Note also that if needed, we might replace the target by $N\times\BR^d$ for whatever $d$ we want, replacing then the function $f$ by the function $(p,z)\to f(p)$. It follows that we might assume without loss of generality that the dimension of $N$ is much larger than that of $M$. 

Now, since everything depends continuously on $\rho$ and $F$ when we perturb in the $C^\infty$-topology, we might suppose without loss of generality that $dF_x$ is injective and that $\grad f\vert_{F(x)}\notin dF_x(T_xM)$. This together with the fact that $\rho >0$ locally implies that, up to reducing $M$ to a neighborhood of $x$, we might assume that the map 
$$(-\epsilon,\epsilon)\times M\to N,\ (t,x)\mapsto\phi_t(F(x))$$ 
is an embedding. This has two important consequences:

\begin{enumerate}
    \item First, there is a function $\hat\rho\in C^\infty(N)$ with 
$$\hat\rho(\phi_t(F(x)))=\rho(x)\text{ for all }(t,x)\in(-\epsilon,\epsilon)\times M$$
and hence with 
$$\frac{\D}{\D t}F_t(x)=-\hat\rho(F_t(x))\cdot\grad(f)\vert_{F_t(x)}.$$
\item Second, whenever $X\in\Gamma(TM)$ is a vector field with $X\vert_x=X_x$ then there is a global vector field $Y\in\Gamma(TN)$ with 
$$Y\vert_{F_t(x)}=dF_t\vert_x(X\vert_x)\text{ for all }(t,x)\in(-\epsilon,\epsilon)\times M$$
Note that the Lie bracket $[Y,-\hat\rho\cdot\grad (f)]_{F_t(x)}=0$ vanishes for all $(t,x)\in(-\epsilon,\epsilon)\times M$, and hence that 
$$\DD_{\frac{\D}{\D t}F_t(x)}Y=\DD_Y\left(-\hat\rho\cdot\grad(f)\right)\vert_{F_t(x)},$$
again for all such $(t,x)$.
\end{enumerate}

All what is left is a calculation. Taking all derivatives at $t=0$ we have:
\begin{align*}
\frac d{dt}\Vert dF_t\vert_xX_x\Vert^2
&=\frac d{dt}\Vert Y\vert_{F_t(x)}\Vert^2=\left(\frac{\D}{\D t} F_t(x)\right)\langle Y,Y\rangle\\
&=2\cdot\langle\DD_{\frac{\D}{\D t} F_t(x)}Y,Y\rangle=2\cdot\langle\DD_Y\left(-\hat\rho\cdot\grad (f)\right),Y\rangle\\
&=-2\cdot\langle\hat\rho\cdot\DD_Y\grad f+(Y\hat\rho)\cdot\grad(f),Y\rangle\\
&=-2\cdot\hat\rho (Hf\vert_{F(x)})(Y_{F(x)},Y_{F(x)})-2\cdot d\hat\rho(Y_{F(x)})\cdot df_{F(x)}(Y_{F(x)}).
\end{align*}
Recalling that $Y_{F(x)}=dF_xX_x$ and that $\rho=\hat\rho\circ F$ we get that $d\hat\rho(Y_{F(x)})=d\rho(X_x)$, and the claim follows.
\end{proof}

\subsection{Energy and its first variation}
Recall that if $V$ and $W$ are Euclidean vector spaces, that is vector spaces endowed with a scalar product, then there is a unique scalar product on $\Hom(V,W)$ with the property that if $(v_1,\dots,v_n)$ and $(w_1,\dots,w_m)$ are orthonormal bases of $V$ and $W$, then $(v_i^*\otimes w_j)_{i,j}$ is an orthonormal basis of $\Hom(V,W)\simeq V^*\otimes W$ where $(v_1^*,\dots,v_n^*)$ is the dual basis to $(v_1,\dots,v_n)$. The {\em energy} of a linear map $L:V\to W$ is nothing other than the square of its norm 
$$e(L)=\Vert L\Vert^2_{\Hom(V,W)}=\sum_{i}\Vert L(v_i)\Vert^2$$
with respect to this scalar product. Before moving on, note also that the isomorphism $V^*\simeq V$ given by the scalar product induces an isomorphism $\Bil(V\times V,\BR)\simeq\Hom(V,V)$, meaning that we can interpret bi-linear forms as endomorphisms. This explains why we refer to the quantity $\tr(H)=\sum_i H(v_i,v_i)$ as the {\em trace} of $H\in\Bil(V\times V,\BR)$. 
\medskip

Moving on now to the world of manifolds, recall that the {\em total energy}, or just simply the {\em energy}, of a smooth map $F:M\to N$ between two Riemannian manifolds is the integral
$$E(F)=\int_M e(dF_p)d\vol_M(p)$$
of the pointwise energy $e(dF_p)$ of the differential $dF_p:T_pM\to T_{F(p)}N$. Here, the integral is taken with respect to the Riemannian measure on the domain $M$ of $F$. We will be interested in how the energy changes when we perturb a map $F:M\to N$ as in Lemma \ref{lem distortion flow}. The following follows directly from the said lemma:

\begin{lem}[First variation of energy]\label{lem first variation energy}
Let $F:M\to N$ be a smooth map between Riemannian manifolds, let $f\in C^\infty(N)$ and $\rho\in C^\infty_c(M)$ be two smooth functions, the latter non-negative and with compact support, and for $t\in\BR$ set 
$$F_t:M\to N,\ F_t(x)=\phi_{\rho(x)\cdot t}(F(x))$$
where $(\phi_t)$ is the local flow of $-\grad f$. We have
$$\frac d{dt}E(F_t)\vert_{t=0}=-2\int_M\tr\left(\rho\cdot F^*(Hf)+d\rho\otimes F^*(df)\right)_pd\vol_M(p).
\eqno\qed$$ 
\end{lem}

The statement of Lemma \ref{lem first variation energy} is rendered more complicated than it would be desirable by the presence of the compactly supported function $\rho$: if $M$ is compact, and we take $\rho\equiv 1$, then it simplifies to 
$$\frac d{dt}E(F_t)\vert_{t=0}=-2\int_N\tr(F^*(Hf))_pd\vol_M(p).$$
It thus follows that if $F$ is non-constant and $Hf$ is positive definite, that is if $f$ is strictly convex, then $\frac d{dt}E(F_t)\vert_{t=0}<0$, meaning that $F$ is not a critical point for the energy functional. Said differently, {\em if $M$ is compact and $F:M\to N$ is non-constant, then the energy decreases if we flow $F$ in the direction of the negative gradient of a smooth strictly convex function}. 

\subsection{Energy of holomorphic maps and the Wirtinger inequality}
A well known fact, key for us, is that holomorphic maps between K\"ahler manifolds are critical points of the energy functional. This follows basically from the {\em Wirtinger inequality} (see \cite{Eells-Sampson}):

\begin{named}{Wirtinger inequality}
Let $M$ and $N$ be K\"ahler manifolds and denote by $\omega_M$ and $\omega_N$ their K\"ahler forms. For any smooth finite energy map $F:M\to N$ we have
$$E(F)\ge\int_M (F^*\omega_N)\wedge\overbrace{\omega_M\wedge\dots\wedge\omega_M}^{\dim_\BC M-1}$$
with equality if and only if $F$ is holomorphic.
\end{named}

Note that in the case that the domain $M$ is a Riemann surface, the Wirtinger inequality simplifies to $E(F)\ge\int_M F^*\omega_N$, once again with equality if and only if $F$ is holomorphic.
\medskip

Anyways, the key observation now is that, since the K\"ahler forms $\omega_M$ and $\omega_N$ are closed, we get from Stokes' theorem that 
$$\int_M(F_t^*\omega_N)\wedge(\omega_M^{\dim_\BC M-1})=\int_M(F^*\omega_N)\wedge(\omega_M^{\dim_\BC M-1})$$ 
for any compactly supported deformation $F_t$ of $F$. This observation, together with the Wirtinger inequality, implies that holomorphic maps are critical points of the energy functional. Indeed, if $M$ and $N$ are K\"ahler and if $F_t:M\to N$ is a compactly supported perturbation of a holomorphic map $F:M\to N$ then, applying twice the Wirtinger inequality we have
$$E(F)=\int_M(F^*\omega_N)\wedge(\omega_M^{\dim_\BC M-1})=\int_M(F_t^*\omega_N)\wedge(\omega_M^{\dim_\BC M-1})\le E(F_t)$$
with equality if and only if $F_t$ is also holomorphic. We record this fact:

\begin{prop}\label{prop wirtinger applied}
Let $M$ and $N$ be K\"ahler manifolds with K\"ahler forms $\omega_M$ and $\omega_N$ and let $F:M\to N$ be holomorphic. Then we have
$$E(F_t)\ge E(F)$$
for any compactly supported perturbation $F_t$. Moreover, we have equality if and only if $F_t$ is also holomorphic.\qed
\end{prop}

We refer to \cite{Lawson} for more results along the lines discussed here and many beautiful applications. 

\subsection{Non-existence of certain holomorphic maps}
In this section we prove Theorem \ref{thm non-existence}, which we restate here for convenience of the reader:

\begin{named}{Theorem \ref{thm non-existence}}
Suppose that $N$ is a K\"ahler manifold whose metric is dominated by a multiple of the Kobayashi metric and which admits a strictly convex function with subexponential growth. Let also $M$ be either a closed connected K\"ahler manifold or an irreducible quasi-projective variety. Then there are no non-constant holomorphic maps $F:M\to N$.
\end{named}

Let us get a pesky issue directly out of the way: while every convex function on a manifold without boundary is continuous and even locally Lipschitz, such functions do not need to be smooth. However, in the course of the proof of the theorem, we might assume without loss of generality that the convex function $f:N\to\BR$ is actually smooth. Indeed, Greene and Wu \cite{Greene-Wu} proved that every strictly convex function on a non-necessarily complete manifold can be approximated by smooth strictly convex functions. They moreover prove that if the original function is $L$-Lipschitz then the approximating functions can, for any $\epsilon>0$, be taken to be $(L+\epsilon)$-Lipschitz. What their argument indeed shows is that if the original function is $L$-Lipschitz on some set, then the approximating function can be taken to be $(L+\epsilon)$-Lipschitz on that set. It follows that if the original strictly convex function has subexponential growth, then the approximating smooth strictly convex function can be chosen to also have subexponential growth. In other words, when proving Theorem \ref{thm non-existence} we might assume without loss of generality that the function $f$ is smooth. After this comment, we are now ready to prove the theorem.

\begin{proof}
Since this is the case we really care about, we will prove the claim in the case that $M$ is a quasi-projective variety. The argument for a closed K\"ahler manifold follows the same lines and is actually a bit simpler---we leave it to the reader.

Well, seeking a contradiction suppose that $F:M\to N$ is a non-constant holomorphic map where $M\subset\BC P^k$ is a quasi-projective variety of dimension $d$ and note that there is some copy $L\subset\BC P^k$ of $\BC P^{k-d+1}$ such that the restriction of $F$ to $L\cap M$ is not locally constant: every 1-dimensional complex subspace of $T_pM$ is the tangent space of such an intersection. The intersection $L\cap M$ might well be singular, but it is dominated by some Riemann surface of finite analytic type. It follows that it suffices to prove the claim in the case that the domain $M$ is a Riemann surface of finite analytic type. From now on, we work in this specific setting.

Although it is not strictly necessary for the argument, note also that we can remove a few points of $M$ to ensure that $F:M\to N$ has maximal rank everywhere and that $M$ has negative Euler characteristic. This latter property implies that $M$ admits a complete conformal hyperbolic metric, which is automatically K\"ahler because $M$ has complex dimension one. Note moreover that, since we are assuming that the metric of $N$ is dominated by the Kobayashi metric, we get that the holomorphic map $F:M\to N$ is Lipschitz.

After all these preparations, we come to the meat of the argument. Let $f:N\to\BR$ be our strictly convex function of subexponential growth. As we discussed just prior to the proof, we might assume without loss of generality that $f$ is smooth. Let then $(\phi_t)$ be the (local) flow associated to $-\grad f$, let $\rho\in C^\infty_c(M)$ be a compactly supported smooth function, and consider for small $t>0$ the compact deformation of $F$ given by
$$F_t:M\to N,\ F_t(p)=\phi_{\rho(p)\cdot t}(F(p))$$
We get from the Wirtinger inequality that $E(F_t)\ge E(F)$. Invoking Lemma \ref{lem first variation energy} we thus get a contradiction when we show that we can choose $\rho\in C^\infty_c(M)$ with
\begin{equation}\label{eq what we want}
\int_M\tr\left(\rho\cdot F^*(Hf)+d\rho\otimes F^*(df)\right)_pd\vol_M(p)>0.
\end{equation}
The remaining of the proof is devoted to construct such a function $\rho$.

To begin with, note that from the assumption that $f$ is smooth and strictly convex, we get that the Hessian $Hf\vert_p$ is positive definite at every $p\in M$. We in particular get that $\Tr(F^*(H_f))_x>0$ for every $x\in M$, and hence that the integral
\begin{equation}\label{eq big integral}
\int_M\rho\cdot\Tr(F^*(H_f))_pd\vol_M(p)>0
\end{equation}
is positive. We will not care whether the integral is finite or not.

Now recall that each cusp of $M$ has a standard neighborhood, which abusing terminology we call simply {\em a cusp}, isometric to the quotient 
$$U=\{z\in\BH^2\text{ with }\Im(z)\ge 1\}/(z\sim z+1)$$
where $\BH^2$ is the hyperbolic plane. Fix $\beta:[0,1]\to[0,1]$ a smooth bump function which sends a neighborhood of $0$ (resp.\ $1$) to $1$ (resp.\ $0$) and consider for $L\ge 0$ the function
$$B_L:U\to[0,1],\ B_L(z)=\left\{\begin{array}{ll}
1 & \text{ if }\Im(z)\le e^L\\
\beta(\frac 1L\log(\Im(z))-1) & \text{ if }\Im(z)\in[e^L,e^{2L}]\\
0 & \text{ otherwise}
\end{array}\right.$$
If $V\subset M$ is a cusp, that is the standard neighborhood of a cusp, then we will also denote by $B_L$ the function on $V$ obtained by composing the isometry $V\simeq U$ with the actual function $B_L$. With this notation, consider a smooth and compactly supported function
$$\rho_L:M\to[0,1],\ \rho_L(z)=\left\{\begin{array}{ll}
B_L(z) & \text{ if }z\text{ belong to a cusp}\\
1 & \text{ otherwise}
\end{array}\right.$$
The norm of $d\rho_L$ is bounded independently of the point and independently of $L$ and this implies that there is a constant $C$ with
$$\vert\Tr(d\rho_L\otimes F^*df)\vert_p\vert\le C\cdot\Vert F^*df\vert_p\Vert$$
for every $L$ and every point $p\in M$. Picking a base point $p_0\in M$ outside the cusps, note that $d\rho_L\vert_p=0$ for all $p\in M$ with $d_M(p,p_0)\le L$. Combining these two facts, and denoting by $B^M(p_0,L)$ the ball in $M$ of radius $L$, we get
$$\left\vert\int_M\Tr(d\rho_L\otimes F^*df)\vert_pd\vol_M(p)\right\vert\le C\cdot\int_{M\setminus B^M(p_0,L)}\Vert F^*df\vert_p\Vert\, d\vol_M(p)$$
for all $L$. Now, the assumption that $f$ has subexponential growth and the fact that $F$ is $1$-Lipschitz imply that $\Vert F^*df\vert_p\Vert$ grows subexponentially, meaning in particular that for every $\epsilon\in(0,1)$ there is some other constant $c$ with 
$$\Vert F^*df\vert_p\Vert\le c\cdot e^{(1-\epsilon)\cdot d_M(p_0,p)}.$$
Up to replacing $C$ by $C\cdot c$ we thus get for all $L$ that
\begin{align*}
\left\vert\int_M\Tr(d\rho_L\otimes F^*df)\vert_pd\vol_M(p)\right\vert
&\le C\cdot\int_{M\setminus B^M(p_0,L)}e^{(1-\epsilon)\cdot d_M(p_0,p)}\, d\vol_M(p)\\
&=C\cdot \int_L^\infty e^{(1-\epsilon)t}\ell_M(\D B^M(p_0,t))\ dt.
\end{align*}
Here $\ell_M(\D B^M(p_0,L))$ is the length (or if you so wish, the 1-dimensional Hausdorff measure) of the set of points in $M$ which are exactly at distance $L$ from $p_0$, and the integral is with respect to the standard Lebesgue measure on $\BR$. Noting that there is some $c>0$ with $\ell_M(\D B^M(p_0,t))\le c\cdot e^{-t}$ for all $t$. Replacing again $C$ by $C\cdot c$  we get that 
$$\left\vert\int_M\Tr(d\rho_L\otimes F^*df)\vert_p d\vol_M(p)\right\vert\le C\cdot \int_L^\infty e^{-\epsilon\cdot t} dt\to 0\text{ as }L\to\infty.$$
Since we also evidently have
$$\lim_L\int_M\rho_L\cdot\Tr(F^*Hf)\vert_pd\vol_M(p)\to \int_M\Tr(F^*(H_f))_pd\vol_M(p)\text{ as }L\to\infty,$$
we get from \eqref{eq big integral} that 
$$\int_M\tr\left(\rho_L\cdot F^*(Hf)+d\rho_L\otimes F^*(df)\right)_pd\vol_M(p)>0$$
for all large enough $L$. It follows that for any such $L$, the function $\rho=\rho_L$ satisfies \eqref{eq what we want}. This concludes the proof of Theorem \ref{thm non-existence}.
\end{proof}

Note at this point that, as long as the target $N$ is a `good' orbifold then the statement of Theorem \ref{thm non-existence} also holds true in the category of orbifolds, where an orbifold is {\em good} if it has a manifold as a finite cover. Let us explain why. First note, as in the proof, that it suffices to consider the case that $M$ is a Riemann surface of finite analytic type. Now, the assumption that there is a manifold $N'$ and a finite orbifold cover $\pi:N'\to N$ implies that any map $F:M\to N$ lifts to a map $F':M'\to N'$ from a finite cover $M'$ of $M$. Evidently, if $N$ satisfies that conditions in Theorem \ref{thm non-existence}, and if we lift the structure of $N$ to $N'$, then so does $N'$. Moreover, if $F$ is holomorphic then so is $F'$ and if $f:N\to\BR$ is strictly convex and has subexponential growth, then the function $f\circ\pi:N'\to\BR$ has those same properties. The theorem, as stated, implies thus that the lifted holomorphic map $F'$ is constant, from where we get that also $F$ is constant, as we wanted to prove.

To conclude, let us stress that in the proof of Theorem \ref{thm non-existence} we did not use completeness of $N$.

\section{Main results}\label{sec main}
The main goal of this section is to prove Theorem \ref{main} and Theorem \ref{thm irreducible} from the introduction.

\subsection{Irreducibility}
Recall that we are thinking of the moduli space as the complex orbifold
$$\CM_{g,r}=\CT_{g,r}/\Map_{g,r}$$
and that maps $F:M\to\CM_{g,r}$ are always in the category of orbifolds: they are induced by maps $\tilde F:\tilde M\to\CT_{g,r}$ on the universal cover of $M$ which are equivariant under some homomorphism $F_*:\pi_1(M)\to \Map_{g,r}$. 

\begin{bem}
    Although we are not making it explicit with our notation, if $M$ is a good orbifold instead of a manifold, we denote by $\pi_1(M)$ the fundamental group in the orbifold category. 
\end{bem}

Theorem \ref{thm irreducible} asserts that, under suitable conditions on $M$, the homomorphism $F_*:\pi_1(M)\to \Map_{g,r}$ induced by a non-constant holomorphic map $F:M\to\CM_{g,r}$ is {\em irreducible} in the sense that its image $F_*(\pi_1(M))$ is not contained in the stabilizer of any simple multicurve $\gamma$ in the surface $S_{g,r}$ of genus $g$ and with $r$ punctures. We restate the theorem for convenience of the reader:

\begin{named}{Theorem \ref{thm irreducible}}
If $M$ is an irreducible quasi-projective variety and $F:M\to\CM_{g,r}$ is a non-constant holomorphic map, then the homomorphism $F_*:\pi_1(M)\to \Map_{g,r}$ is irreducible.
\end{named}

\begin{proof}
Suppose that the orbifold $M$ is a quasi-projective variety with universal cover $\tilde M$ and fundamental group $\pi_1(M)$ and that $F:M\to\CM_{g,r}$ is a non-constant holomorphic map. Seeking a contradiction, suppose that the image of $F_*(\pi_1(M))\subset \Map_{g,r}$ fixes a multicurve $\gamma$. This implies that $F_*(\pi_1(M))\subset\Stab_{\Map_{g,r}}(\gamma)$ and hence that $F$ lifts to a map 
$$F':M\to\CT_{g,r}/\Stab_{\Map_{g,r}}(\gamma).$$
We endow the target with the Weil-Petersson metric, which we recall is K\"ahler and dominated by a multiple of the Kobayashi metric. As we pointed out after the proof of Theorem \ref{thm non-existence}, the theorem holds true in the category of orbifolds, as long as the target is a good orbifold, that is an orbifold which is finitely covered by a manifold. The orbifold $\CT_{g,r}/\Stab_{\Map_{g,r}}(\gamma)$ is good because $\Map_{g,r}$ has a finite index torsion free subgroup. 

It follows from all of this that to be able to get a contradiction from Theorem \ref{thm non-existence} we just need to exhibit a strictly convex function with subexponential growth on $\CT_{g,r}/\Stab_{\Map_{g,r}}(\gamma)$. Luckily for us, we get from Proposition \ref{prop 3} that the function
$$h_\gamma:\CT_{g,r}\to\BR_{\ge 0}$$
defined in \eqref{eq height function} is strictly convex, has subexponential growth and is invariant under $\Stab_{\Map_{g,r}}(\gamma)$. It thus descends to a strictly convex function 
$$\hat h_\gamma:\CT_{g,r}/\Map_{g,r}(\gamma)\to\BR_{>0}$$
with sub-exponential growth. Having found our function, we just got a contradiction to the assumption that $F_*(\pi_1(M))$ fixes $\gamma$. We are done.
\end{proof}

\begin{bem}
Note that the same statement and proof apply if $M$ is a general, but closed, K\"ahler manifold. Remark also that the reader that feels uneasy with using the fact that moduli space is a quasi-projective variety could avoid it as follows. In the proof of Theorem \ref{thm non-existence} we only used that the domain was quasi-projective to reduce to the case that it was a Riemann surface. We did that using that in quasi-projective varieties there are plenty of algebraic curves. In the specific setting of moduli space, one can instead use the fact that there are plenty of Teichm\"uller curves.
\end{bem}

\subsection{Homomorphisms between mapping class groups}

Our next and last goal is to prove Theorem \ref{main}, but first we must recall a few things about homomorphisms between mapping class groups. In general, group homomorphisms $\Map_{g,r} \to \Map_{g',r'}$ can be quite diverse. Some well-known maps are induced by covers, inclusions, and forgetting punctures. Let us give some examples:

\begin{bei}[Automorphisms]\label{example aut}
Ivanov proved that all automorphisms of the mapping class group $\Map_{g,r}$ are of the form $[\psi]\mapsto[\Psi\circ\psi\circ\Psi^{-1}]$ where $\Psi$ is a homeomorphism of $S_{g,r}$. Here and in the sequel $[\cdot]$ stands for the class of. Anyways, if $\Psi$ is orientation preserving, then the map 
$$\CT_{g,r}\to\CT_{g,r},\ [\phi:S_{g,r}\to X]\mapsto[\phi\circ\Psi^{-1}:S_{g,r}\to X]$$
is equivariant under this homomorphism and hence descends to the map $\CM_{g,r}\to\CM_{g,r}$. This map is however not very interesting: it is just the identity. When $\Psi$ is orientation reversing, things are more intriguing. In this case one also has an equivariant map, namely
$$\CT_{g,r}\to\CT_{g,r},\ [\phi:S_{g,r}\to X]\mapsto[\phi\circ\Psi^{-1}:S_{g,r}\to\bar X]$$
where $\bar X$ is the Riemann surface with the complex conjugated structure. The induced map $\CM_{g,r}\to\CM_{g,r}$ depends on the individual element $\Psi$, but in all cases we have that $\Psi$ is not holomorphic, but rather anti-holomorphic.
\end{bei}

\begin{bei}[Forgetful maps]\label{example forgetful}
Think of $S_{g,r}$ as being $S_{g,0}$ with $r$ marked points and, accordingly, think of $\Homeo_+(S_{g,r})$ as a subgroup of $\Homeo(S_{g,0})$. For some $r'\le r$, let $S_{g,r'}$ be obtained from $S_{g,r}$ by forgetting $r-r'$ marked points. The inclusion $\Homeo_+(S_{g,r})\subset\Homeo_+(S_{g,r'})$ induces a homomorphism $\Map_{g,r}\to\Map_{g,r'}$. It is namely the homomorphism associated to the forgetful map $\CM_{g,r}\to\CM_{g,r'}$, which we recall is holomorphic. 
\end{bei}

\begin{bei}[Constructions via covers]\label{example covers} For some $g\ge 2$, let $*\in S_{g, 0}$ be a base point and let $\Gamma\subset\pi_1(S_{g,0},*)$ the subgroup consisting of elements which are trivial in the $\BZ/2\BZ$-homology. The subgroup $\Gamma$ is the fundamental group of a closed surface, and indeed it follows from the Riemann-Hurwitz formula that we have $\Gamma\simeq \pi_1(S_{(g-1)\cdot 2^{2g}-1,0},*)$. Noting that $\Aut(\pi_1(S_{g,0},*))$ preserves $\Gamma$ we get thus a homomorphism 
$$\Aut(\pi_1(S_{g,0},*))\to\Aut(\Gamma)\simeq \Aut(\pi_1(S_{(g-1)\cdot 2^{2g}-1,0}),*)).$$ 
Since $\Map_{g,1}\simeq\Aut(\pi_1(S_{g,0},*))$ we get a homomorphism
$$\Map_{g,1}\to\Map_{(g-1)\cdot 2^{2g}-1,1}$$
By construction, this homomorphism is associated to a holomorphic map $\CM_{g,1}\to\CM_{(g-1)\cdot 2^{2g}-1,1}$. Indeed, there is an equivariant map $\CT_{g,1}\to\CT_{(g-1)\cdot 2^{2g}-1,1}$ given by lifting each complex structure on $S_{g,1}$ to a complex structure on $S_{(g-1)\cdot 2^{2g}-1,2^{2g}}$ and then forgetting all but one of the marked points. 

In the absence of marked points it is a bit harder to use covers to get homomorphisms between mapping class groups, but it was shown in \cite{paper with chris} that for some suitable $g'$ there is a cover $\pi:S_{g',0}\to S_{g,0}$ which induces a homomorphism $\Map_{g,0} \to \Map_{g',0}$. This homomorphism is once again induced by the holomorphic map $\CM_{g,0}\to\CM_{g',0}$ given lifting to $S_{g',0}$ via $\pi$ complex structures on $S_{g,0}$.
\end{bei}

\begin{bei}[Multi-embeddings]\label{example multi-embedding}
Thinking now of $S_{g,r}$ as an open surface, that is as a surface with cusps instead of marked points, suppose that for some $(g',r')$ we have an (automatically finite) collection $\CI=\{\iota_i: S_{g,r} \to S_{g',r'}\}$ of embeddings with disjoint image. Every multi-embedding induces a (diagonal) homomorphism 
$$\CI_*:\Homeo_c(S_{g,r})\to\Homeo_c(S_{g',r'})$$
between the groups of compactly supported homeomorphisms of the domain and the target.
If the homomorphism $\CI_*$ induces a homomorphism $\CI_*:\Map_{g,r}\to\Map_{g',r'}$, then we say that the latter is {\em induced by the multi-embedding} $\CI$. 

It is important to keep in mind that not every multi-embedding induces a homomorphism between mapping class groups. Indeed, \cite[Lemma 2.19]{Rodrigo} asserts that a multi-embedding $\CI$ as above induces a homomorphism between mapping class groups if and only if for every $i$ and every oriented curve $\gamma\subset S_{g,r}$ which can be homotoped into a cusp, one of the following holds:
\begin{itemize}
\item[c1)] Either $\iota_i(\gamma)$ is homotopically trivial,
\item[c2)] or $\iota_i(\gamma)$ can be homotoped into a cusp of $S_{g',r'}$, 
\item[c3)] or there is another $i'\neq i$ with $\iota_{i'}(\gamma)$ homotopic to $\iota_i(\gamma)$ such that the two embeddings $\iota_i$ and $\iota_{i'}$ pullback opposite orientations.
\end{itemize}
For example, let $\Sigma$ be a compact surface with interior homeomorphic to $S_{g,1}$ and identify $S_{2g,0}$ with the double of $\Sigma$. The two tautological embeddings of $S_{g,1}=\Sigma\setminus\D\Sigma$ into the double yield a multi-embedding $\CI$ of $S_{g,1}$ satisfying the conditions above and hence inducing a homomorphism $\CI_*:\Map_{g,1}\to\Map_{2g,0}$. Noting that this homomorphism is not irreducible, we get from Theorem \ref{thm irreducible} that $\CI_*$ is not induced by any holomorphic (or anti-holomorphic) map $\CM_{g,1}\to\CM_{2g,0}$.
\end{bei}

As the reader might have noticed, there is some redundancy in the above list of examples: both automorphisms and homomorphisms induced by forgetful maps are also induced by multi-embeddings. Indeed, a homomorphism $\CI_*$ induced by a multi-embedding $\CI$ is the composition of an automorphism and a forgetful map if and only if the collection $\CI$ is just an embedding, in the sense that it has a single element. 

Anyways, the reason we stress homomorphisms induced by multi-embeddings is that in \cite{Rodrigo} it is proved that in the range of Theorem \ref{main}, all non-trivial homomorphisms between mapping class groups are induced by some multi-embedding:

\begin{sat}\cite[Theorem 1.2]{Rodrigo}\label{thm rodrigo}
    Let $g\geq 4$ and $g'\leq 3\cdot 2^{g-3}$. Every non-trivial homomorphism $\varphi:\Map_{g,r} \to \Map_{g',r'}$ is induced by a multi-embedding.
\end{sat}

\begin{bem}
We wish to stress that the result proved in \cite{Rodrigo} is quite more general, allowing for surfaces which not only have cusps but also boundary components.
\end{bem}

As we mentioned earlier, some homomorphisms induced by multi-embeddings are reducible and hence, by Theorem \ref{thm irreducible}, don't come from any holomorphic map between moduli spaces. The following proposition characterizes the irreducible homomorphisms $\varphi:\Map_{g,r}\to\Map_{g',r'}$ in the range we are interested in:

\begin{prop}\label{prop homomorfismos irred}
A homomorphism $\varphi: \Map_{g,r} \to \Map_{g',r'}$ induced by a multi-embedding is irreducible if and only if it is the composition of an automorphism and of a homomorphism induced by a forgetful map.
\end{prop}
\begin{proof}
If $\varphi$ is the composition of automorphisms and  forgetful maps, then it is easy to see that $\varphi$ is irreducible. We show the other direction. 

Let $\CI$ be a multi-embedding inducing $\varphi$. Observe that $\varphi$ is the composition of automorphisms and forgetful maps if and only if $\CI$ contains a single embedding. Also, note that if $S$ is closed, then every multi-embedding contains a single homeomorphism. Thus, to prove the statement we show $\varphi$ is reducible if $S$ is a punctured surface and $\CI$  contains at least two embeddings.

If $\CI$ contains two embeddings, then $g'>g$. Even more, by \cite[Lemma 2.19]{Rodrigo} there is a curve $\gamma \subset S_{g,r}$ homotopic to a cusp and an embedding $\iota\in\CI$ such that $\iota(\gamma)\subset S_{g',r'}$ is a non-trivial curve non-homotopic to a cusp. Since every element of $\Map_{g,r}$ fixes the homotopy class of $\gamma$ and $\varphi$ is induced by a multi-embedding, then the image of $\varphi$ fixes the homotopy class of $\iota(\gamma)$. In other words, $\varphi$ is reducible. 
\end{proof}

\subsection{The resolution}
We come now to the final act of this paper, the proof of Theorem \ref{main}:

\begin{named}{Theorem \ref{main}}
Suppose that $g\ge 4$ and that $g'\le 3\cdot 2^{g-3}$, and let $r,r'\ge 0$. Then every non-constant holomorphic map $F:\CM_{g,r}\to\CM_{g',r'}$ is a forgetful map. In particular, if such a map exists, then $g'=g$ and $r'\le r$.
\end{named}

\begin{proof}
Let $F_*: \Map_{g,r} \to \Map_{g',r'}$ be the homomorphism corresponding to a non-constant holomorphic map $F: \CM_{g,r} \to \CM_{g',r'}$. From Theorem \ref{thm irreducible} we get that $F_*(\Map_{g,r})$ does not fix any non-trivial simple multi-curve. It thus follows from Theorem \ref{thm rodrigo} and Proposition \ref{prop homomorfismos irred} that $F_*$ is induced by the composition of an automorphism $$\Map_{g,r}\to\Map_{g,r}, [\psi]\mapsto[\Psi\circ\psi\circ\Psi^{-1}]$$ 
as in Example \ref{example aut} and of a forgetful homomorphism 
$$\Map_{g,r}\to\Map_{g',r'}$$
as in Example \ref{example forgetful}. In particular, $g=g'$ and $r'\leq r$. Moreover, since the composition of an anti-holomorphic map and of a holomorphic map is anti-holomorphic we get that the homeomorphism $\Psi\in\Homeo(S_{g,r})$ inducing our automorphism has to be orientation preserving. From here, it follows that $F_*:\Map_{g,r}\to\Map_{g,r'}$ after conjugacy agrees with the homomorphism $\hat{F}_*$ induced by the forgetful map $\hat{F}:\CM(S_{g,r})\to \CM(S_{g,r'})$. Now, since Teichm\"uller space is a classifying space for proper actions of the mapping class group \cite{Ji Wolpert} we get that $F$ and $\hat F$ are homotopic. To conclude the proof, it suffices to invoke a result from \cite[Proposition 3.2]{griego}: any two non-constant holomorphic and homotopic maps from a quasi-projective variety to the moduli space agree. We are done.
\end{proof}

\end{document}